\documentclass[11pt]{article}

\usepackage{amsmath,amsthm,amsfonts,amssymb}
\usepackage{graphicx}
\usepackage[usenames]{color}
\usepackage{hyperref}
\usepackage{soul}

\newtheorem{theorem}{Theorem}[section]

\newtheorem{lemma}[theorem]{Lemma}
\theoremstyle{definition}

\newtheorem{corollary}[theorem]{Corollary}

\newtheorem{definition}[theorem]{Definition}

\newtheorem{proposition}[theorem]{Proposition}

\newcommand{\fracd}[2]{{\displaystyle\frac{#1}{#2}}}
\newcommand{\abs}[1]{\vert #1 \vert}

\newcommand{\bi}{\begin{itemize}}
\newcommand{\ei}{\end{itemize}}
\newcommand{\ben}{\begin{enumerate}}
\newcommand{\een}{\end{enumerate}}

\title{Regular sets and counting in free groups}

\author{Elizaveta Frenkel\\ Alexei G. Myasnikov\\ Vladimir N. Remeslennikov}

\begin{document}
\maketitle



\tableofcontents

\section{Introduction}

In this paper we study asymptotic behavior of regular subsets in a
free group $F$ of finite rank, compare their sizes at infinity, and
develop  techniques to compute the probabilities of sets relative to
distributions on $F$ that come naturally from random walks on the
Cayley graph of $F$. We apply these techniques to study cosets,
double cosets, and Schreier representatives of finitely generated
subgroups of $F$ with an eye on complexity  of algorithmic  problems
in free products with amalgamation and HNN extensions of groups.

During the last decade it has been realized that a natural set of algebraic objects usually can be divided into two parts. The large one (the {\em regular part}) consists of  typical, "generic" objects; and the smaller one (the "singular" part)  is made of "exceptions". Essentially, this idea  appeared first in the form of zero-one laws in probability  theory, number theory, and combinatorics. It became popular after seminal works of Erd\H{o}s,  that shaped up the so-called Probabilistic Method (see, for  example, \cite{AS}). In finite group theory the idea of genericity can be traced down to a series of papers by Erd\H{o}s and Turan in 1960-70's (for recent results see, for example,  \cite{shalev}).  In combinatorial group theory the concept of generic behavior is due to Gromov. His  inspirational  works \cite{Gromov1,Gromov2} turned the subject into an  area of very active research, see for example,  \cite{AO,A1,A2,multiplicative,BogV,BV,BMS,CET,CPAI,BMR2,BM,BV,Champetier1,Champetier2, Jitsukawa,  KMSS1,KMSS2,KSS,KS1,KS2,KRSS,KR,Ollivier,Olsh,Rom,Enric,Zuk}. It turned out that the generic objects usually have  much simpler structure, while the exceptions provide most of the difficulties. For instance, generic finitely generated groups are hyperbolic \cite{Gromov1,Olsh}, generic subgroups of hyperbolic groups are  free \cite{GMO}, generic cyclically reduced elements in free groups are of minimal length in their automorphic orbits \cite{KRSS}, generic automorphisms of a free group are strongly irreducible  \cite{Rivin}, etc.

   In practice, the generic-case analysis of algorithms is usually  more important than the worst-case one. For example, knowing generic properties of objects one can often design simple algorithms that work very fast on generic elements. In cryptography, many successful attacks  exploit  the generic properties of random elements from a particular class, ignoring the existing hard instances \cite{MU,MSUbook,MSU,RST}.  In the precise form the  generic complexity of algorithmic problems appeared first in the papers \cite{KMSS1,KMSS2,multiplicative,BMS}. We refer the reader to a comprehensive survey \cite{GMMU} on generic complexity of algorithms.

 In this paper we lay down some techniques that allow one to measure sets which appear naturally when computing with infinite finitely presented groups. Our main idea is to approximate a given set by some regular subsets and estimate the asymptotic sizes of  the regular sets using powerful tools of  random walks on graphs and generating functions. The particular applications we have in mind concern with the generic complexity of the Word and Conjugacy problems in free products with amalgamation and HNN extensions. In general, such  problems  can be extremely hard. In  \cite{miller1} Miller described a free product of free groups with  finitely generated amalgamation
 where the  Conjugacy problem is undecidable; while in \cite{miller1} he
 gave similar examples in the class of  generalized
HNN-extension of  free groups. However,  it has been proven  in \cite{CPAI,BMR2} that  on a precisely described set $RP$ of "regular elements" in amalgamated free products and HNN extensions $G$ the Conjugacy problem is decidable (under some natural conditions on the factors), furthermore,  it is decidable in polynomial time.   Namely,
it was shown in \cite{CPAI,BMR2} that
the group $G$ (satisfying some natural assumptions) can be stratified into two parts with respect to the
``hardness'' of the conjugacy problem:
 \begin{itemize}
\item  the  \emph{Regular Part} $RP$ consists of so-called {\em
regular elements} for which the conjugacy problem is decidable in polynomial time by the
 standard algorithms (described in  \cite{MKS,LS}). Moreover, one can decide whether or not a given element is regular in $G$;
 \item  the {\em Black Hole} $BH$ (the complement of $RP$ in $G$)
  consists of elements in $G$ for
 which either the standard algorithms  do not work at all, or they are slow, or the situation is not quite clear yet.

  \end{itemize}

 The missing piece is to show that the set $RP$ is, indeed, generic in $G$. This is not easy, the complete proof, which  will appear in \cite{fmrIII}, relies on the techniques developed in the present paper.
Now, a few words on the structure of the paper. In Section \ref{Section:preliminaries}, following  \cite{multiplicative}, we describe  some  techniques  for measuring subsets
in  a free group $F$, the asymptotic classification of large and small  sets, and approximations via context-free  and regular sets.

 In Section \ref{section_Systems} we study, using graph techniques,
  Shreier system of representatives (transversal) of a finitely generated subgroup $C$ in  a free group $F$ of finite rank.
  If $S$ is a fixed Schreier transversal of $C$ then $s \in S$   is called {\em stable} (on the right) if $sc \in
S$ for any $c \in C$. Intuitively, the stable representatives are "regular", they are easy to deal with.

In Section  \ref{section_measurable} we estimate the sizes of various subsets of $F$. In particular, we show that $S$ is regular and thick (see definitions in Section \ref{Section:preliminaries}), meanwhile the set $S_{nst}$ of non-stable representatives from $S$ is exponentially negligible. Furthermore, the set $S_{nst}$ is exponentially negligible even relative to the set $S$. Our approach here is to "approximate" the sets in question by regular sets and to measure sizes of the regular sets using tools of random walks on graphs and Perron-Frobenius techniques.

In Section \ref{Section:theorem_spherical} we develop a technique to compare sizes of different regular sets at "infinity" and  give an asymptotic classification of  regular
subsets of $F$ relative to a fixed prefix-closed regular subset $L
\subseteq F$. The main result describes when regular subsets of $L$ are "large" or "small" at infinity in comparison  to $L$. Notice, in the case when $L = F$, this result has been proven in \cite{multiplicative} (Theorem 3.2).

\section{Preliminaries}\label{Section:preliminaries}

In this section, following  \cite{multiplicative}, we describe  some  techniques  for measuring subsets
in  a free group $F$, the asymptotic classification of large and small  sets, and approximations via context-free  and regular sets.

\subsection{Asymptotic densities}\label{subs_measure}

 Let $F = F(X)$ be a free group with basis $X=\{x_1,\dots,
x_m\}$. We use this notation throughout the paper.

Let $R$ be a subset of the free group $F$ and $S_k = \{\, w \in F
\mid |w|=k\,\}$ the sphere of radius $k$ in $F$. The fraction
\[f_k(R)= \frac{\vert R \cap S_k\vert}{\vert S_k\vert}
\]
is the {\em  frequency} of elements from $R$ among the words of
length $k$ in $F.$  The {\em asymptotic density}
 $\rho(R)$ of $R$ is defined by
  $$\rho(R) = \limsup_{k \rightarrow \infty} f_k(R).$$
$R$ is called {\em generic}  if
 $\rho(R) = 1$,  and {\em negligible} if $\rho(R) = 0$.
 If, in addition, there
exists a positive constant $\delta <1$ such that
$$1-\delta^k < f_k(R) < 1$$
 for all sufficiently large $k$
then $R$ is called {\em exponentially generic}. Meanwhile, if
$ f_k(R) < \delta^k$ for large enough $k$ then $R$ is {\em exponentially negligible}. In both the cases we refer to $\delta$ as a {\em rate upper bound}.
 Sometimes such sets are also  called \emph{strongly generic} or  \emph{strongly
negligible}, but we refrain from this.

The {\em Cesaro limit}
 \begin{equation}
 \label{eq:cesaro-1}
  \rho^c(R)  = \lim_{n\rightarrow\infty} \frac{1}{n}\left(f_1+\cdots+f_n\right).
  \end{equation}
  gives another asymptotic  characteristic, called  {\em Cesaro density}, or
    {\it asymptotic average density}.  Sometimes, it is more sensitive then standard asymptotic density $\rho$ (see,
for example, \cite{multiplicative}, \cite{woess}). However, if $\lim_{k\rightarrow \infty} f_k(R)$
exists (hence is equal to $\rho(R)$) then $\rho^c(R)$ also exists
and $\rho^c(R) = \rho(R)$. We will have to say more about $\rho^c(R)$ below.

Asymptotic density gives the first coarse classification of large (small) subsets:

\paragraph{Coarse classification}
\begin{itemize}
\item [1)] {\em Generic } sets;
\item [2)] {\em Visible} or {\em thick} sets: the set $R$ is {\em visible} if $\rho(R)> 0$;
\item [3)] {\em Negligible} sets.
\end{itemize}
Unfortunately, this classification is very coarse, it does not distinguish many sets which, intuitively, have different sizes.

All our results in this paper concern with the strong version of the asymptotic density $\rho$, when the actual limit $\lim_{k\rightarrow\infty}f_k(R)$ exists. This allows one to differentiate sets with the same asymptotic density with respect to their growth rates.   Thus generic sets $R$ divide into subclasses of {\em exponential, subexponential, superpolynomial, polynomial} generic sets, with respect  to the convergence rates of their frequency  sequences $\{f_k(R)\}_{k\in \mathbb{N}}$. The same holds for negligible sets as well.




\subsection{Generating random elements and multiplicative measures}
\label{subsec:random-generator}

  One can use a no-return random walk
$W_s$ ($s \in (0,1]$) on the Cayley graph $C(F,X)$ of $F$ with respect
to the generating set $X$, as a random generator of elements of $F$ (see \cite{multiplicative}).
 We start at the identity element $1$ and
either do nothing with probability $s$ (and return value $1$ as the
output of our random word generator), or move to one of the $2m$
adjacent vertices with equal probabilities $(1-s)/2m$. If we are at a
vertex $v \ne 1$, we either stop at $v$ with probability $s$ (and
return the value $v$ as the output), or move, with probability
$\frac{1-s}{2m-1}$, to one of the $2m-1$ adjacent vertices lying away
from $1$, thus producing a new freely reduced word $vx_i^{\pm 1}$. Since
the Cayley graph $C(F,X)$ is a tree and we never return to the
word we have already visited,
it is easy to
see that the probability $\mu_s(w)$ for our process to terminate at a word $w$
is given by the formula
\begin{equation}
\mu_s(w) = \frac{s(1-s)^{|w|}}{2m\cdot (2m-1)^{|w|-1}} \quad
\hbox{ for } w \ne 1
\end{equation}
and
\begin{equation}
\mu_s(1) =s.
\end{equation}

  For $R \subseteq F$ its measure $\mu_s(R)$ is defined by  $\mu_s(R)= \sum_{w\in R} \mu_s(w)$.
Recalculating $\mu_s(R)$ in terms of $s$, one gets
 $$ \mu_s(R) = s\sum_{k=0}^\infty f_k(1-s)^k, $$
 and the series on the right hand side is convergent for all $s \in
(0,1)$. The ensemble of distributions $\{\mu_s\}$ can be encoded in a single  function
$$\mu(R): s \in (0,1)  \rightarrow \mu_s(R) \in {\mathbb{R}}.$$
The argument above shows that for every subset $R \subseteq F$,
$\mu(R)$ is an analytic function of $s$.

 It has been shown in \cite{multiplicative} that $\mu(R)$ contains a lot of information about the asymptotic behaviour of the set $R$.  To see where this information comes from  renormalise the  measures $\mu_s$  and
consider the parametric  family $\mu^* = \{\mu_s^* \}$ of  {\em
adjusted measures}
\begin{equation}
\mu_s^*(w) = \left(\frac{2m}{2m-1} \cdot \frac{1}{s} \right)\cdot
\mu_s(w).
\end{equation}
This new measure $\mu_s^*$  is {\it multiplicative} in the sense
that
\begin{equation}
\mu_s^*(u\circ v) = \mu_s^*(u)\mu_s^*(v),
\end{equation}
where $u\circ v$ denotes the product of non-empty words $u$ and
$v$ such that $|uv| = |u| +|v|$ (no cancelation in the product $uv$).
  Moreover, if we denote
\begin{equation}
t = \mu_s^*(x_i^{\pm 1})=  \frac{1-s}{2m-1} \label{eq:adjusted}
\end{equation}
 then
\begin{equation}
\mu_s^*(w) = t^{|w|}
\end{equation}
 for every non-empty word $w$.  Assume now, for the sake of minor
technical convenience, that $R$ does not contain the identity
element $1$. It is easy to see that
$$
\mu^*_s(R) = \sum_{k=0}^\infty n_k(R)t^k $$
is  the generating function of the spherical growth
sequence
$$n_k(R) =  |R \cap S_k|$$
 of the set $R$  in variable $t$ which is convergent for
each  $t \in [0,1)$.

 The distribution $\mu_s$ has the uncomfortably big standard
deviation $\sigma = \frac{\sqrt{1-s}}{s}$, which reflects the fact
that $\mu_s$  is strongly skewed towards "short" elements.
The mean length of words in $F$
distributed according to $\mu_s$ is equal to $L_s = \frac{1}{s}
-1$, so $L_s \to \infty$ when $s \to 0$. This shows that the asymptotic behaviour of the set $R$ at "infinity" (when   $L_s \rightarrow \infty$)  depends on the behaviour of the function
$\mu(R)$ when $s\rightarrow 0^+$.

Following \cite{multiplicative}, for a subset $R$ of  $F$ we define a numerical characteristic
 $$ \mu_0(R) = \lim_{s \rightarrow 0^+}\mu(R) =  \lim_{s \rightarrow 0^+} s \cdot \sum_{k=0}^\infty
f_k(1-s)^k.$$
If $\mu(R)$ can be expanded  as a convergent power series in $s$
at $s=0$ (and hence in some neighborhood of $s = 0$):
 $$ \mu(R) = m_0 + m_1s + m_2s^2 + \cdots, $$ then
  $$\mu_0(R) = \lim_{s
\rightarrow 0^+} \mu_s(R) = m_0,$$
and an easy corollary from a theorem by Hardy
and Littlewood \cite[Theorem~94]{Hardy}
asserts that $\mu_0(R)$ is precisely the Cesaro limit $\rho^c(R)$.

A subset $R \subseteq F$ is called {\em smooth} \cite{multiplicative} if $\mu(R)$ can be expanded  as a convergent power series in $s$ at $s=0$.

\subsection{The frequency measure}
\label{subsec:freq-measure}

In this section we discuss the frequency measure, introduced in \cite{multiplicative}.

 Let $W_0$ be the  no-return
non-stop random walk  on the Cayley graph $C(F,X)$ of $F$ (like $W_s$ with $s =
0$), where the walker moves from a given vertex  to any adjacent
vertex away from the initial vertex 1 with equal probabilities
$1/2m$. In this event, the probability $\lambda(w)$ that the
walker hits an element $w \in F$ in $|w|$ steps
(which is the same as the probability
that the walker ever hits $w$)  is equal to
$$
\lambda(w) = \frac{1}{2m(2m-1)^{|w|-1}}, \ \hbox{ if } \ w \neq 1, \ \
\hbox{ and }  \ \lambda(1) = 1.
$$

 This gives rise to a  measure called  the {\em frequency } measure on $F$, or {\em Boltzmann distribution},
 defined for subsets $R \subseteq F$ by
$$ \lambda(R) = \sum_{w \in R}\lambda(w),$$
if the sum above is finite. One can view $\lambda(R)$ as  the {\em cumulative frequency}  of $R$ since
$$\lambda(R) = \sum_{k=0}^{\infty}f_k(R).$$
This measure  is not probabilistic,
since, for instance, $\lambda(F) = \infty$, moreover, $\lambda$ is additive, but not $\sigma$-additive.

 A subset $R \subseteq F$ is called {\em $\lambda$-measurable}, or simply {\em  measurable} (since we do not consider any other measures in this paper) if
$\lambda(R) < \infty$. Every measurable set is negligible.

\paragraph{Linear approximation.}
 If the set $R$ is smooth then  the linear term in the expansion  of
$\mu(R)$ gives a linear approximation of  $\mu(R)$:
$$
\mu(R) = m_0 + m_1s + O(s^2).
$$
 In this case,  $m_0 = \mu_0(R)$ is the Cesaro density of $R$.
An easy corollary of \cite[Theorem~94]{Hardy} shows
 that if $\mu_0(R) = 0$ then
$$ m_1 = \sum_{k=1}^\infty f_k(R) = \lambda(R). $$

On the other hand, even without assumption that $R$ is smooth,
 if $R$ is measurable, then
 $$\mu_0(R) = 0 \ \  \hbox{ and }\ \ \mu_1 = \lim_{s\rightarrow 0^+} \frac{\mu(s)}{s} =  \lambda(R).$$

\subsection{Asymptotic classification of subsets}
  \label{subsec:asymptotic}

In this section we describe   a classification of subsets $R$ in
$F$,  according to the asymptotic behavior of the functions
$\mu(R)$.

Recall, that the function
$\mu(R)$ is analytic on  $(0,1)$ for every subset $R$ of $F$.  $R$ is
{\em smooth} if $\mu(R)$ can be analytically extended to a
neighborhood of $0$. The subset $R$ is called
{\it rational, algebraic, etc,  } with respect to $\mu$  if the
function $\mu(R)$ is  rational, algebraic, etc.

\paragraph{Asymptotic classification of sets.}
  The following  subtler classification of sets in $F$
 (based on the linear approximation of $\mu(R)$) was introduced in \cite{multiplicative}:
\begin{itemize}
\item {\em Thick subsets}: $\mu_0(R)$ exists, $\mu_0(R) > 0$ and
$$\mu(R) = \mu_0(R) + \alpha_0(s), \ \ where \ \ \lim_{s \rightarrow
0^+}\alpha_0(s) = 0.$$

\item {\em Negligible subsets of intermediate density}: $\mu_0(R) = 0$
 but $\mu_1(R)$ does not exist.

 \item {\em Sparse negligible subsets}: $\mu_0(R) = 0$, $\mu_1(R)$ exists  and
$$\mu(R) = \mu_1(R)s + \alpha_1(s)\ \  where \ \ \lim_{s \rightarrow
0^+}\frac{\alpha_1(s)}{s}  = 0.$$

\item {\em Exponentially negligible sets:}
\item {\em Singular  sets}: $\mu_0(R)$ does not exist.

\end{itemize}

For sparse sets, the values of $\mu_1$ provide a further
and more subtle discrimination by size.

\begin{lemma} \cite{multiplicative}
A subset  is sparse in $F$ if and only if it is measurable.
\label{lm:lambda=sparse}
\end{lemma}

\subsection{Context-free and regular languages as a measuring tool}
\label{subsec:context-free}

The simple observation in  Section  \ref{subsec:random-generator} that $\mu(R)$ is the generating function of the grows sequence $\{n_k(R)\}_{k\in \mathbb{N}}$ allows one   to apply a well established machinery of
generating functions of regular and context-free languages to estimate asymptotic
sizes  of subsets $R$ in $F$. We refer to \cite{HU} on regular and context-free languages, and to \cite{eps} on regular languages in groups.

\paragraph{Algebraic sets and context free languages.}
If the set $R$ is an (unambiguous) context free language then, by
a classical theorem of Chomsky and Schutzenberger \cite{Ch-Sch},
the generating function $\mu^*(R) = \sum n_k(R)t^k$, and hence the
function $\mu(R)$, are algebraic functions of $s$. Moreover, if
$R$ is regular  then $\mu(R)$ is a rational function with
rational coefficients \cite{flajolet,stanley}.

It is well known that singular points of an algebraic function
are either poles
or branching points. Since $\mu(R)$ is bounded for
$s \in (0,1)$, this means that,
for a context-free set $R$, the function $\mu(R)$
has no singularity at $0$ or has a branching
point at $0$. A standard result on analytic functions allows us
to expand $\mu_s(R)$
 as a fractional power
series:
 $$ \mu_s(R) = m_0 + m_1 s^{1/n} + m_2s^{2/n} + \cdots ,$$
$n$ being the branching index. This technique was used in \cite{BMR2,BM} for numerical estimates of generic complexity of algorithms.

If $R$ is regular, than we actually have the usual power series expansion:
 $$ \mu_s(R) = m_0 + m_1 s + m_2s^{2} + \cdots; $$
 in particular,  $\mu(R)$ can be analytically extended in the
 neighborhood of $0$, so $R$ is smooth.

 The following gives an asymptotic  classification of regular subsets of $F$.
 \begin{theorem}\cite{multiplicative,af_semr}
\label{th:regular-negl-thick}
\bi
\item [1)] Every negligible regular subset of $F$ is strongly
negligible.
\item [2)] A regular subset of $F$ is thick if and only if its prefix closure contains a cone.
\item [3)] Every regular subset of $F$ is either thick or strongly
negligible. \ei
\end{theorem}

\section{Schreier Systems of
Representatives}\label{section_Systems}

\subsection{Subgroup and coset graphs}\label{subs_graphs}
In this section for a given finitely generated subgroup of a free group we discuss its subgroup and coset graphs.

Let $F = F(X)$ be a free group with basis $X=\{x_1,\dots,
x_n\}$.  We identify  elements of $F$ with  reduced words  in the alphabet $X \cup X^{-1}$.
Fix a subgroup  $C = \langle h_1, \ldots, h_m \rangle$ of $F$
  generated by finitely many elements $h_1, \ldots, h_m  \in F$.

Following  \cite{km}, we associate with $C$ two graphs: \emph{the
subgroup graph} $\Gamma =\Gamma_C$ and the coset graph $\Gamma^\circ
= \Gamma^\circ_C$. We freely  use definitions and results from
\cite{km} in the rest of the paper.

Recall, that $\Gamma$  is a finite connected digraph with edges
labeled by elements from $X$ and a distinguished vertex
(based-point) $1$,  satisfying  the following two conditions.
Firstly,  $\Gamma$ is folded, i.e., there are no two edges in
$\Gamma$ with the same label and having  the same initial or
terminal vertices. Secondly, $\Gamma$ accepts precisely the reduced
words in $X \cup X^{-1}$  that belong to $C$. To explain the latter
observe, that walking along a path $p$ in $\Gamma$ one  can read a
word $\ell(p)$ in the alphabet $X \cup X^{-1}$, the label of $p$,
(reading $x$ in  passing an edge $e$ with label $x$ along the
orientation of $e$, and reading $x^{-1}$ in the opposite direction).
We say that $\Gamma$ accepts a word $w$ if $w = \ell(p)$ for some
closed path $p$ that starts at $1$ and has no backtracking. One can
describe $\Gamma$ as a deterministic  finite state  automata with
$1$ as the  unique  starting and accepting state.

For example, the graph $\Gamma$ for the subgroup
generated by $x_1x_2x_1^{-1}$ is shown in the picture below.

\begin{center}
\setlength{\unitlength}{.5mm}
\begin{picture}(100,40)(-20,0)
\thicklines \put(19,18){$1$} \put(20,25){\circle*{2}}
\put(20,25){\vector(1,0){38}} \put(38,26){$x_1$}
\put(60,25){\circle*{2}}
\put(60.5,24){\vector(-1,3){0}} \put(70,25){\circle{20}}
\put(82,25){$x_2$}
\end{picture}\\
Pic. 1.
\end{center}

Given the generators $h_1, \ldots, h_m$ of the subgroup $C$, as words from $F(X)$, one can effectively construct the graph $\Gamma$ in time $O(n\log^\ast n)$ \cite{Touikan}.

The {\em  coset graph} (also known as the Schreier graph) $\Gamma^\ast = \Gamma_C^\ast$ of $C$  is a connected labeled digraph with
 the set $\{Cu \mid u \in F\}$ of the right cosets of $C$ in $F$ as the vertex set, and such that there is an edge from
 $Cu$ to $Cv$ with a label $x \in X$ if and only if $Cux = Cv$.  One can describe the coset graph $\Gamma^\ast$ as  obtained from $\Gamma$
by the following procedure. Let
$v \in \Gamma$ and $x \in X$ such that there is no outgoing or incoming
 edge at $v$ labeled by $x$. We call such $v$ a {\em boundary}  vertex of $\Gamma$ and denote the set of such vertices
 by $\partial \Gamma$. For every such $v \in \partial \Gamma$ and $x \in X$ we attach
to $v$ a new edge $e$ (correspondingly, either outgoing or incoming)  labeled $x$ with a new terminal vertex $u$ (not in $\Gamma$). Such
vertices $u$ are called {\em frontier} vertices, we denote the set of frontier vertices of $\Gamma$ by $\partial^+ \Gamma$. Then we
 attach to $u$ the Cayley graph $C(F,X)$ of $F$ relative to $X$ (identifying $u$ with the vertex $1$ of $C(F,X)$), and
 then we fold the edge $e$ with the corresponding edge in $C(F,X)$ (that is labeled $x$ and is incoming to $u$).
Observe, that for every vertex $v \in \Gamma^\ast$ and every reduced word $w$ in $X \cup X^{-1}$ there is a unique path $\Gamma^\ast$
that starts at $v$ and has the label $w$.  By $p_w$ we denote such a path that starts at $1$, and by $v_w$ the end vertex of $p_w$.
 Here is
the fragment of the
 graph $\Gamma^\ast$ for  $C = \left< x_1x_2x_1^{-1}\right>$:

\begin{center}
\setlength{\unitlength}{.5mm}
\begin{picture}(100,120)(-20,-40)
\thicklines \put(22,26){$1$} \put(20,25){\circle*{2}}
\put(20,25){\vector(1,0){38}} \put(20,65){\vector(1,0){18}}
\put(20,65){\vector(0,1){18}} \put(0,65){\vector(1,0){18}}

\put(38,26){$x_1$} \put(60,25){\circle*{2}}
\put(60,25){\vector(1,0){18}}
\put(62.5,24){\vector(-3,1){0}} \put(60,15){\circle{20}}
\put(58,1){$x_2$} \put(68,26){$x_1$} \put(80,25){\circle*{2}}
\put(80,25){\vector(1,0){18}} \put(80,25){\vector(0,1){18}}
\put(80,5){\vector(0,1){18}}

\put(-20,25){\circle*{2}} \put(20,25){\vector(0,1){38}}
\put(20,-15){\vector(0,1){38}} \put(20,-15){\circle*{2}}
\put(20,-15){\vector(1,0){18}} \put(0,-15){\vector(1,0){18}}
\put(20,-35){\vector(0,1){18}}

\put(22,5){$x_2$} \put(22,45){$x_2$}
\put(-20,25){\vector(1,0){38}} \put(-2,26){$x_1$}

\put(-20,25){\vector(0,1){18}} \put(-40,25){\vector(1,0){18}}
\put(-20,5){\vector(0,1){18}}
\end{picture}\\
Pic. 2.
\end{center}

 \begin{lemma}  \label{vnr:5.1}
  $\Gamma^\ast_C $ is the coset graph of $C$ in $F$.
\end{lemma}
 \begin{proof} See, for example, \cite{km}. \end{proof}
Notice that $\Gamma = \Gamma^\ast$ if and only if the subgroup $C$
has finite index in $F$. Indeed, $\Gamma = \Gamma^\ast$  if and
only if  for every vertex $v$ of $\Gamma$ and every label $x\in
X$, there is an edge in $\Gamma$ labeled by $x$ which exits from
$v$, and an edge with label $x$ which enters $v$, but this is
precisely the characterization of  subgroups of finite index in
$F$ \cite[Proposition~8.3]{km}.

A spanning  subtree $T$ of $\Gamma$  with the root at the vertex 1
is called {\em geodesic}  if for every vertex $v \in V(\Gamma)$
the unique path in $T$ from 1 to $v$  is a geodesic path in
$\Gamma$. For a given graph $\Gamma$ one can effectively construct
a geodesic spanning subtree $T$ (see, for example, \cite{km}).

From now on we fix an arbitrary  spanning  subtree $T$ of
$\Gamma$. It is easy to see that the tree $T$ uniquely extends to a spanning subtree
$T^\ast$ of $\Gamma^\ast$.

Let $V(\Gamma^\ast)$ be the set of vertices of $\Gamma^\ast$. For
a subset $Y \subseteq V(\Gamma^\ast)$ and a subgraph $\Delta$ of
$\Gamma^\ast$, we define the \emph{language accepted by} $\Delta$
and $Y$ as the set $L(\Delta, Y,1)$ of the labels $\ell(p)$ of  paths $p$ in $\Delta$ that start at $1$ and end
at one of the vertices in $Y$, and have no backtracking. Notice
that the  words $\ell(p)$ are reduced since
the graph $\Gamma^\ast$ is folded. Notice, that $F = L(\Gamma^{\ast}, V(\Gamma^{\ast}), 1)$ and  $C = L(\Gamma,\{1\},1) =  L(\Gamma^{\ast}, \{1\}, 1)$.

Sometimes we will refer to a set of right (left) representatives
of $C$ as the right (left) {\em transversal} of $C$. Furthermore,
to simplify terminology, a  transversal will usually mean a right
transversal, if not said otherwise. Recall, that a  transversal $S$ of $C$ is
termed {\em Schreier} if  every initial segment of a representative from $S$  belongs to $S$.

\begin{proposition}
\label{prop:coset-graph} Let $C$ be a finitely generated subgroup
of $F$. Then:

\ben

 \item [1)] for every spanning subtree $T$ of $\Gamma$ the set $S_{T^\ast} = L(T^{\ast}, V(\Gamma^{\ast}), 1)$ is  a
Schreier transversal of $C$ .
 \item [2)] for every Schreier transversal $S$ of $C$ there exists a spanning
 subtree $T$ of $\Gamma$ such that $S =
 S_{T^\ast}$.

  \een
\end{proposition}
 \begin{proof}
The statement 1) follows directly from Lemma \ref{vnr:5.1}.
To prove 2) notice that
every reduced path $p \in \Gamma^{\ast}$ can be decomposed as  $p = p_{int}
\circ p_{out},$ where $p_{int}$ is a maximal reduced path in
$\Gamma,$ and $p_{out}$ is the tail of $p$ outside  of
$\Gamma.$ This decomposition is unique. Moreover:

 \bi
 \item if $v \in V(\Gamma)$ and $p$ is a reduced path
from $1$ to $v$ in $\Gamma^{\ast}$ then $p$ passes only through
vertices of $\Gamma;$
\item if  $v \in V(\Gamma^{\ast})
\setminus V(\Gamma)$ and $p^\prime$ and $p^{\prime \prime}$ are
two paths from $1$ to $v,$ where $p^\prime = p_{int}^{\prime }
\circ p_{out}^{\prime},$ and $p^{\prime \prime} = p_{int}^{\prime
\prime} \circ p_{out}^{\prime \prime}$, then $p_{out}^{\prime} =
p_{out}^{\prime \prime}$.
 \ei
 Let $S$ be  a  Schreier transversal of $C$ in $F$ and $s \in S$. Suppose that the reduced path $p_s$ ends at some vertex $v_s$ in $\Gamma$.  Then the whole path $p_s$
 lies in $\Gamma.$ Let $T$ be  a subgraph of
$\Gamma$ generated by the union of all paths $p_s$, where $s \in S$ and $v_s \in \Gamma$. Since
$S$ is a Schreier transversal  $T$ is a maximal subtree of
$\Gamma.$ It is clear that $S=S_{T^{\ast}}$.  Hence, the result. \end{proof}

Proposition \ref{prop:coset-graph} allows one to identify elements
from a given Schreier transversal $S$ of $C$ with the vertices of
the graph $\Gamma^\ast$, provided  a maximal subtree of $\Gamma$ is fixed.
We use this frequently in the
sequel.

\begin{corollary}\label{co:number-transv}
The number of distinct Schreier transversals of $C$ in $F$ is finite and equal to the number of spanning subtrees of $\Gamma_C$.
\end{corollary}

\subsection{Schreier transversals}\label{subs_spec_syst}

In this section we introduce various types of representatives of $C$ in $F$ relative to a fix basis $X$ of $F$.

\begin{definition} Let $S$ be a transversal of $C$.
\begin{itemize}

\item
 A representative $s\in S$ is
called \emph{internal} if the path $p_s$ ends in $\Gamma$, i.e., $v_s \in V(\Gamma)$.
 By $S_{\rm int}$ we denote the
set of all internal representatives in $S.$  Elements from $S_{\rm
ext}=S \smallsetminus S_{\rm int}$ are called the \emph{external}
representatives in $S.$
\item A representative $s \in S$ is called {\em geodesic} if it
has minimal possible length in its coset $Cs$. The transversal $S$
is {\em geodesic }  if every $s \in S$ is geodesic.

\item A representative $s \in S$ is called {\em singular} if it belongs to the generalized normalizer  of $C$:
$$N^*_F(C) = \{ f \in F | f^{-1} C f \cap C \neq 1 \}.$$
 All other representatives
from $S$ are called {\em regular}. By $S_{\rm sin}$ and,
respectively, $S_{\rm reg}$ we denote the sets of singular and
regular representatives from $S$.

\item A representative $s \in S$ is called {\em stable} (on the right) if $sc \in
S$ for any $c \in C$.  By $S_{\rm st}$ we denote the set of all
stable representatives in $S,$ and $S_{\rm uns} = S \smallsetminus
S_{\rm st}$ is the set of all non-stable representatives from $S$.
\end{itemize}
\end{definition}

In the following lemma we collect some basic properties of various
types of representatives. Recall that the  {\em cone} defined by (or
based at) an element $u \in F$ is the  set  $C(u)$ of all reduced
words in $F$ that have  $u$ as an initial segment.  For $u, v \in F$
we write $u \circ v$ if there is no cancelation in the product $uv$,
i.e., $|uv| = |u| + |v|$. In this case $C(u) = \{w \in F \mid w =
u\circ v, v \in F\}$.

\begin{proposition}
\label{pr:basic-properties} Let $S$ be a Schreier transversal for
$C$, so $S = S_{T^{\ast}}$ for some spanning subtree ${T^{\ast}}$ of
$\Gamma^{\ast}$. Then the following hold:
  \bi
   \item [1)] $S_{\rm int}$ is a basis of $C$, in particular, $|S_{\rm int}| = |V(\Gamma)|.$
    \item [2)] $S_{\rm ext}$ is  the union of finitely many coni $C(u),$ where $v_u \in \partial^+
    \Gamma$.

    \item  [3)]$S_{\rm sin}$ is contained in a finite union of double
    cosets $Cs_1s_2^{-1}C$ of $C$, where $s_1, s_2 \in S_{\rm int}$.
    \item [4)] $S_{\rm uns}$ is a finite union of  left
    cosets of $C$ of the type $s_1s_2^{-1}C$, where $s_1, s_2 \in S_{\rm
    int}$.
 \ei
\end{proposition}

\begin{proof} 1) is well-known, see \cite{bmr}, for example.  2) follows immediately from the construction.
To see  3) notice first that  $S_{\rm sin} \subseteq N^*_F(C)$ and
$N^*_F(C)$ is the union  of finitely many double cosets $CsC$, where
$s \in S_{\rm sin}$, and furthermore, every such coset has the form
$CsC = C s_1 s_2^{-1} C,$ where $s_1, s_2 \in S_{\rm int}$ (see
Lemma 5 in \cite{bmr}, or Propositions 9.8 and 9.11 in \cite{km}, or
Theorem 2 in \cite{Jitsukawa}).

 To see 4) assume that $s \in S$ is
not stable, so there exists an element $c \in C$ such that $sc \not
\in S$.  Then $s= s_1 \circ t, c = t^{-1} \circ d, sc = s_1 \circ
d.$  We claim, that the terminal vertex of $s_1$ lies in $\Gamma$ (viewing $s$ as a path in $\Gamma^\ast$).
 Indeed, if not, then $s_1$, as well as $s_1 \circ d$, is in $S_{\rm ext}$ - contradiction.
 Hence, $s_1 \in S_{\rm int}$. Since $c = t^{-1}\circ d$ there is a closed path in $\Gamma$ with the label $t^{-1}\circ d$, starting at $1_C$. Let $s_2 \in S_{\rm int}$ be the representative of $t^{-1}$. Then $s_2d = c_1 \in C$, hence $sc = s_1\circ d = s_1s_2^{-1}c_1$, so $s \in  s_1s_2^{-1}C$, as claimed.

 \end{proof}

\begin{proposition}
\label{pr:basic-connections} Let $S$ be a Schreier transversal for
$C$, so $S = S_{T^{\ast}}$ for some spanning subtree ${T^{\ast}}$ of
$\Gamma^{\ast}$. Then the following hold:
  \bi
   \item [1)] If $T^{\ast}$ is a geodesic subtree of $\Gamma^{\ast}$ (and hence $T$ is a geodesic subtree of $\Gamma$) then $S$ is
   a geodesic transversal.
   \item [2)] If $C$ is a malnormal subgroup of $F$  then $S_{\rm sin} = \emptyset.$
    \item  [3)]$S_{\rm sin} \subseteq S_{\rm uns}.$
  \ei
\end{proposition}
\begin{proof} 1) is straightforward (see also \cite{km}).

2) If $C$ is malnormal then $N^*_F(C)=1$,  so $S_{\rm sin} = \emptyset.$

3) If  $s \in S_{\rm sin}$ then  $c = s^{-1} c_1 s$ for some
non-trivial $c, c_1 \in C$, so
 $c_1s = sc$. Since $sc \neq s$  we conclude that  $sc \not \in S$, hence $s \in S_{\rm uns}.$ \end{proof}

\section{Measuring subsets of $F$}\label{section_measurable}

Recall that a \emph{finite automaton}
$\mathcal{A}$ is a finite labeled oriented graph (possibly
with multiple edges and loops). We refer to its vertexes as
\emph{states}. Some of the states are called \emph{initial}
states, some \emph{accept} or \emph{final} states. We assume
further that every edge of the graph is labeled by one of the
symbols $x^{\pm 1}, x\in X,$ where $F = F(X)$ is a  free group of
finite  rank $m.$
The \emph{language accepted by an automaton} $\mathcal{A}$ is the
set $L= L(\mathcal{A})$ of labels on paths from  initial to
accept states. An automaton is said to be \emph{deterministic} if,
for any state there is at most one arrow with a given label
exiting  the state. A \emph{regular} set is a language
accepted by a finite deterministic automaton.

The following facts
about regular sets are well known. Let $A$ and $B$ are regular
subsets in $F$. Then:

\bi \item the sets $A \cup B,$ $A \cap B$ and $A \smallsetminus
B$ are regular.

\item The prefix closure $\overline{A}$ of a regular set $A$ is
regular. Here, the prefix closure $\overline{A}$ is the set of all
initial segments of all words in $A.$

\item If $ab = a\circ b$ for any $a\in A, b \in B$   then
$AB$ is regular.

\item If $ab = a\circ b$ for any $a, b \in A$ then $A^{\ast}$ is regular.

 \ei

The following notation is useful. For $u, v \in F$ define
$$c(u,v) = \frac{1}{2}(|u| + |v| - |uv|)$$
 - the amount of cancelation in the product $uv$.

\begin{proposition}
\label{le:basic-sparse} Let $R_1$ and $R_2$ be subsets of $F$.
Then the following statements hold:
 \bi
   \item[1)] If $R_1 \subseteq R_2$ and $R_2$ is negligible
   (exponentially negligible) then so is $R_1$.

\item [2)] If $R_1, R_2$  are negligible
(exponentially negligible) then so is $R_1 \cup R_2$.

     \item [3)] If $R_1$ and $R_2$ are negligible (exponentially negligible) then so is the set
    $$R_1 \circ R_2 = \{r_1r_2 \mid r_i \in R_i, \ c(r_1,r_2) = 0 \}.$$

\item [4)] If $R_1$ and $R_2$ are negligible (exponentially negligible) then so is the set
    $$R_1 \mathop{\circ}\limits_t R_2 = \{r_1 r_2 \mid  r_i \in R_i, \ c(r_1,r_2) \leq t \}.$$

 \ei

\end{proposition}
  \begin{proof}
The proof is  straightforward. \end{proof}

\begin{definition}
Let $R_1$ and $R_2$ be subsets of $F$ and $f:R_1 \rightarrow R_2$
a map. Then:
 \bi
  \item $f$  is called {\em d-isometry}, where $d$ is a non-negative real number,  if for any  $w \in R_1$
  $$|w| - d \leq  |f(w)|  \leq |w| + d.$$
   \item $f$  has  {\em uniformly bounded fibers} if there exists a constant $c$
     such that every element $w \in R_2$ has at most $c$ pre-images in $R_1$.
      \ei
\end{definition}
\begin{proposition}
\label{le:basic-sparse-2}
 Let $R_1$ and $R_2$ be subsets of $F$. Then
the following statements hold:
  \bi
    \item [1)] If $f:R_1 \rightarrow R_2$ is a surjective $d$-isometry and $R_1$ is negligible (exponentially negligible)  then  so is  $R_2$.
    \item [2)] If  $f:R_1 \rightarrow R_2$ is a $d$-isometry  with uniformly bounded fibers and  $R_2$ is
    negligible (exponentially negligible) then so is  $R_1$.
 \ei
\end{proposition}
  \begin{proof}
Notice that for $k > d$
$$f_k (R_2) \leq \sum_{j = k-d}^{k+d}f_j(R_1),$$
  and 1) follows. Similarly,
  $$f_k(R_1)\leq c \sum_{j= k-d}^{k+d}f_j(R_2)$$
  for $k > d$ and 2) follows.\end{proof}

\begin{proposition}
\label{pr:S-thick}
 Let  $C$ be  a finitely generated subgroup of
infinite index of a free group $F$. Then every Schreier transversal of $C$ in $F$ is regular and thick.
\end{proposition}
 \begin{proof}
By definition $S = S_{\rm int} \cup S_{\rm ext}$. By Proposition
\ref{pr:basic-properties} the set $S_{\rm int}$ is finite and the
set $S_{\rm ext}$ is a non-empty finite union of cones.  By
Theorem \ref{th:regular-negl-thick}  each cone  in $S_{\rm ext}$ is
thick. Therefore, the set $S_{\rm ext}$,  as well  as the set $S$,
is thick. Clearly, every cone is regular, so is the set $S.$ \end{proof}

\begin{proposition}
 \label{pr:sparse}
 Let $C$ be a finitely generated subgroup of infinite
index in $F$. Then the following  hold:
 \bi
   \item[1)]
 $C$ is exponentially negligible in $F$ and one can find some  upper bound $\delta < 1$ for the growth rate of $C$.

  \item[2)] Every coset of $C$ in $F$ is exponentially negligible in $F$.

 \ei
\end{proposition}
 \begin{proof} 1) follows from the Proposition 1 and Corollary 1 in \cite{af_semr}.

  2) follows from 1) above and 4) from
 Proposition \ref{le:basic-sparse}. \end{proof}

\begin{proposition}
 \label{pr:negligible_degrees}
 Let $C$ be a finitely generated subgroup of infinite
index in $F$. Then the following hold:
 \bi
   \item[1)] $C^* = \mathop{\bigcup}\limits_{f\in F} C^f$ is exponentially negligible in $F$.

  \item[2)] For every $c\in C$ the set conjugacy class $c^F = \{ f^{-1} c f | f\in F\}$ is exponentially negligible in $F$.

 \ei
\end{proposition}
 \begin{proof} The statement 1) has been shown in
 \cite{multiplicative} and also in Proposition 1.10 in
 \cite{averina-frenkel}.  The statement 2) is shown in Proposition 1.11 in
 \cite{averina-frenkel} \end{proof}

\begin{proposition}
\label{pr:union-cosets} Let $C$ be a finitely generated subgroup
of infinite index in $F$ and $S$ is a Schreier transversal of
 $C$ in $F.$ If $S_0 \subseteq S$ is a exponentially negligible subset of $F$ then the set $\mathop{\bigcup}\limits_{s \in S_0}Cs$ is exponentially negligible in
$F.$
 \end{proposition}
 \begin{proof}  By Proposition \ref{pr:basic-properties} $S = S_{\rm int} \cup S_{\rm ext}$,
  where  $S_{\rm int}$ is a  finite set and $S_{\rm ext}$ is a union of finitely many cones $C(u), u \in \partial^+ \Gamma.$
 It suffices to prove the result for $S_0 \cap C(u)$ for a fixed $u \in \partial^+ \Gamma$. To this end we may
 assume from the beginning that $S_0 \subseteq C(u)$. If $s$ is the representative of $u$ in $S$ then every word
 from $C(u)$ contains $s$ as an initial segment. Since $s$ is not readable in $\Gamma$ the amount of cancelation $c(w,t)$ in
  the product $wt$, where $w \in C$ and $t \in C(u)$ does not exceed the length of $s$. Hence
 $$CS_0 = C \mathop{\circ}\limits_{|s|} S_0$$
  and the result follows from the statement 4) of Proposition \ref{le:basic-sparse}.   \end{proof}

\begin{proposition}
\label{pr:double-cosets} Let $A$ and $B$ be finitely generated
subgroups of infinite index in $F$. Then for any $w \in F$ the
double coset $AwB$ is exponentially negligible in $F.$
\end{proposition}
\begin{proof} Observe, that $AwB = AB^{w^{-1}}w$, so by the statement 2
of Proposition \ref{pr:sparse} it suffices to show that
 $AB^{w^{-1}}$ is exponentially negligible. Since $B^{w^{-1}}$ is just
another finitely generated subgroup of infinite index in $F$ one can
assume  from the beginning that $w = 1$. Let $S$ be a geodesic
Schreier transversal for $A$ in $F.$ Then
  $$AB = \bigcup_{s \in S_0} As$$
 for some subset $S_0 \subseteq S.$ By Proposition \ref{pr:union-cosets}  it
 suffices to show that the subset $S_0$
  is exponentially negligible. Since the set $S_{\rm int}$ is finite we may assume that  $S_0 \subseteq S_{\rm ext}$.
  Now we construct an  $r$-isometry  $\alpha : S_0
  \rightarrow B.$ Let $T_A$ be the spanning subtree
  of $\Gamma_A$ such that $S = S_{{T_A}^{\ast}}$ and $T_B$ be a spanning geodesic
  subtree of $\Gamma_B$. Denote by  $d$ the maximum of the diameters of the trees $T_A$ and $T_B.$
  To describe the map $\alpha$ choose an
  arbitrary element $s \in S_0.$ Without loss of generality assume that $\abs{s} \geq d,$ because there are only finitely
  many such $s$ that have smaller length and by Proposition
  \ref{pr:sparse} they will not extremely change asymptotic size of $AB$ since $A$
  of infinite index in $F.$
  Then $as = b$ for some $a \in A$
  and $b \in B$. We claim that there exists an element $b_s \in B$
  such that $|sb_s^{-1}| \leq 2d.$ Indeed, the
  cancelation in the product $as$ is at most  $d$ (see the argument in Proposition \ref{pr:union-cosets}).
   Hence $s$ and $b$
  have a common terminal segment $t$ of length
  at least $|s|-d$ (recall that $|s| \geq d$).  It follows that in the graph
 $\Gamma_B$ there exists a path from some vertex $v$ to $1_B$ with
 the label $t_v.$  Then $b_s = t_vt \in B$ and $|sb_s^{-1}| =
  |st^{-1}t_v^{-1}| \leq   2d$.  Hence  $s$ and $b_s$ has a long common terminal segment and differ only on the initial segment of length at most $2d$. It follows that
  the map $\alpha: s \to b_s$ gives a   $2d$-isometry $\alpha:S_0 \to B$. Notice that $\alpha$ has uniformly bounded fibers. Indeed, if $\alpha(s_1) = \alpha(s_2) = b$ then $s_1$ and $s_2$ differ from $b$, hence from each other, only on the initial segment of length at most $2d$. So there are at most $(2d)^{2|X|}$ such distinct elements.  Since $B$ is exponentially negligible by Proposition \ref{le:basic-sparse-2}  the set $S_0$ is
  also exponentially negligible, as claimed. This proves the result.
Notice, that the property being geodesic for Schreier transversal
$S$ for $A$ in $F$ is not crucial for our prove. Namely, for
arbitrary Shreier transversal $S$ all conclusions can be repeated
with slightly different constant.\end{proof}

Now we can state the main result of the section.

\begin{theorem}\label{th:sing-nst-negligible}
Let $C$ be a finitely generated subgroup of infinite index in $F$
 and $S$ a Schreier transversal for $C$. Then the following hold:
 \bi

     \item[1)] The generalized normalizer $N^*_F(C)$ of
$C$ in $F$  is exponentially negligible in $F$.

\item[2)] The set of singular representatives $S_{\rm sin}$
is exponentially negligible in $F$.

\item[3)] The set $S_{\rm uns}$ of unstable representatives is
exponentially negligible in $F$.
 \ei
\end{theorem}
 \begin{proof}
 To see 1) recall that the generalized normalizer $N^*_F(C)$ of $C$ in $F$
is a finite  union  of  double cosets of $C$ in $F.$ Therefore
$N^*_F(C)$ is exponentially negligible in $F$ by Proposition
\ref{pr:double-cosets}.

2) follows immediately  from 1).

 To prove 3) observe that $S_{\rm uns}$ is a finite union of left
cosets  of $C$ (see 3) in Proposition \ref{pr:basic-properties}). Now the result  follows from   Proposition \ref{pr:sparse}. \end{proof}

Theorem \ref{th:sing-nst-negligible} can be strengthen  as follows.

\begin{corollary}
Let $C$ be a finitely generated subgroup of infinite index in $F.$  Then the sets
$$Sin(C) = \mathop{\bigcup}\limits_{S} S_{\rm sin}, \ \ Uns(C) = \mathop{\bigcup}\limits_{S} S_{\rm uns},$$
where $S$ runs over all Schreier transversals of $C$, are
exponentially negligible.
\end{corollary}
\begin{proof}
By Corollary \ref{co:number-transv} there are only finitely many Schreier transversals of $C$. Now the result follows from Theorem \ref{th:sing-nst-negligible} and Proposition \ref{le:basic-sparse}.
\end{proof}

\section{Comparing sets at infinity}\label{Section:theorem_spherical}

 \subsection{Comparing Schreier representatives}
 \label{subsec:comp-Schreier}

 In this section we give another  version of Theorem \ref{th:sing-nst-negligible}.  To explain
 we need a few definitions.

For subsets $R, L$ of $F$ we define their {\em size ratio} at length $k$ by
   $$f_k(R,L) = \frac{f_k(R)}{f_k(L)} = \frac{|R\cap S_k|}{|L\cap S_k|}.$$
  The size ratio  $\rho(R,L)$ {\em at infinity}
  of $R$ relative to $L$ (or the relative asymptotic density) is defined by
  $$\rho(R,L) = \limsup_{k \rightarrow \infty} f_k(R,L).$$
  By $r_L(R)$ we denote the {\em cumulative size ratio} of $R$ relative to $L$:
  $$
  r_L(R) = \sum_{k=1}^\infty f_k(R,L).
  $$
  We say that $R$ is {\em $L$-measurable}, if $r_L(R)$ is finite.
$R$ is called {\em negligible} relative to $L$  if $\rho(R,L) = 0$.
Obviously, an $L$-measurable set is $L$-negligible.
   A set $R$ is termed {\em exponentially negligible} relative to  $L$ (or  {\em exponentially $L$-negligible}) if
   $f_k(R,L) \leq q^k$ for all sufficiently large $k$.

The following result is simple,  but useful.

\begin{proposition}\label{pr:cone-negl}
Let $R$ be an exponentially negligible set in $F$.

\bi
\item [1)]  For any $w \in F$
 the set $R$ is a exponentially negligible relative to the cone $C(w).$

    \item [2)]  The set  $R$ is exponentially negligible relative to any exponentially generic subset $T$ of $F$.
\ei
\end{proposition}
\begin{proof} Observe, that $f_k(C(w))  = 1/2m(2m-1)^{|w|-1}$ is a constant. Since
$$f_k(R,C(w)) = \frac{f_k(R)}{f_k(C(w))}$$
it follows that $R$ is exponentially negligible relative to $C(w)$. This proves 1).

 To  prove 2) denote by $p$ and $q$ the corresponding rate bounds for $R$ and $T$, so $f_k(R) \leq p^k, f_k(T) \geq 1-q^k$ for sufficiently large $k$.  Then, for such $k$,
$$f_k(R,T) = \frac{f_k(R)}{f_k(T)} \leq \frac{p^k}{1-q^k} =   \left( \fracd{p}{(1 - q^k)^{\frac{1}{k}}}\right)^k.$$
 Since
 $$\mathop{lim}\limits_{k \rightarrow \infty} \fracd{p}{(1 - q^k)^{\frac{1}{k}}} = p$$
it follows that for any $\varepsilon > 0$
$$f_k(R,T) \leq (p+\varepsilon)^k$$
 for sufficiently large $k$, as claimed.

\end{proof}

\begin{corollary}\label{s_nst-ins}
Let $C$ be a finitely generated subgroup of infinite index in $F$ and $S$ a Schreier transversal for $C$.
Then the following  hold:

\bi
    \item[1)] The set of singular representatives $S_{\rm sin}$
     is exponentially negligible in $S$.

\item[2)] The set $S_{\rm uns}$ of unstable representatives is
exponentially negligible in $S$.
 \ei
\end{corollary}
\begin{proof} The statements of this corollary follow immediately from
Theorem \ref{th:sing-nst-negligible} and Propositions
\ref{pr:basic-properties}, \ref{pr:basic-connections} and
\ref{pr:cone-negl}. \end{proof}

\subsection{Comparing regular sets}
 \label{subsec:comp-regular}

In this section we give an asymptotic classification of regular
subsets of $F$ relative to a fixed prefix-closed regular subset $L
\subseteq F$.

For this purpose we are going to describe how one can use a random
walk on the finite automaton $\mathcal{B}$ recognizing regular
subset $R \subseteq L$ similar to the one in Section
\ref{subsec:freq-measure}. It will be convenient to further put
$\mathcal{B}$ to special form consistent to $L.$

Recall  Myhill-Nerode's  theorem on  regular languages (see, for
example, \cite{eps}, Theorem 1.2.9.) For a language $R$ over an
alphabet $A$ consider an equivalence relation $\sim$ on $A^{\ast}$
defined as follows: two strings $w_1$ and $w_2$ are equivalent if
and only if   for each string $u$ over $A$ the words $w_1 u$ and
$w_2 u$ are either simultaneously  in $R$ or not in $R$. Then $R$ is
regular if and only if there are only finitely many
$\sim$-equivalence classes.

Now, let  $R \subseteq L$. Define an  equivalence relation $\sim$ on
$L$ such that
 $w_1 \sim w_2$ if and only if  for each $u \in F$ the following condition holds:
 $w_1u = w_1 \circ u$ and $w_1u \in R$ if and only if $w_2 u = w_2 \circ u$ and $w_2 u \in R.$

The following is an analog of Myhill-Nerode's  theorem  for free
groups.
\begin{lemma}\label{Myh-Ner_group}
Let  $R \subseteq L \subseteq F$ and $L$ prefix-closed and regular.
Then $R$ is regular if and only if there are only finitely many
$\sim$-equivalence classes in $L$.
\end{lemma}
\begin{proof}
 The proof is similar to the original one.
 We give a short  sketch of the most interesting part of it. If the set of the
equivalence classes is finite one can  define an automaton
$\mathcal{B}$ on the set of   equivalence class as states. If $x \in
X \cup X^{-1} $ and  $[w]$ is the equivalence class of some $w$ such
that $w \circ x \in R$ then one connects the state $[w]$ with an
edge labeled by $x$ to the state $[wx]$. The class  $[\varepsilon]$,
where $\varepsilon$ is the empty word, is the initial state, while a
state $[w]$ is  an accepting  state if and only if $w \in R.$ In
this case $L(\mathcal{B})= R.$
\end{proof}

Since $R$ is regular, we suppose that $\mathcal{B}$ as in Lemma
\ref{Myh-Ner_group} and modify it in the next way. Without loss of
generality we can assume that $\mathcal{A}$ is in the normal form,
i.e., it has only one initial state $I$ and doesn't contain
inaccessible states.

   Let $S =[w]$ be a state of $\mathcal{B}$. Denote by
$S^{pr}$ the uniquely defined state in  $\mathcal{A}$ which is the
terminal state of the path with the label $w$ in $\mathcal{A}$,
starting at $[\varepsilon]$.  The state $S^{pr}$ is well-defined, it
does not depend on the choice of $w$. We call $S^{pr}$ the
\emph{prototype}
 of $S$.

Since  $\mathcal{B}$ accepts only reduced words in  $X \cup X^{-1}$
one  can transform $\mathcal{B}$ to a form where the following hold:

\bi
\item [a)]  $\mathcal{B}$ has only one initial state $I$ and one accepting  state $Z$.

\item [b)] For any state $S$ of $\mathcal{B},$ all arrows which
enter $S$ have the same label $x \in X \cup X^{-1}$ and arrows
exiting from $S$ cannot have label $x^{-1}$ (this can be achieved by
splitting the states of $\mathcal{B}$, see Pic. 3). We shall say in
this situation that $S$ \emph{has type $x$.}

\item [c)] For every state $S$ of $\mathcal{B}$ there is a direct
path from $S$ to the accept state $Z.$

\item [d)] There are no
arrows entering the initial state $I.$
\par

 \ei

\begin{figure}
 \setlength{\unitlength}{0.8mm}
\begin{picture}(100,70)(-15,-10)

\thicklines \put(55,19){$\Longrightarrow$}

\put(0,0){\circle{6}} \put(0,20){\circle{6}} \put(0,40){\circle{6}}
\put(20,20){\circle{6}} \put(40,10){\circle{6}}
\put(40,30){\circle{6}}

\put(80,0){\circle{6}} \put(80,20){\circle{6}}
\put(80,40){\circle{6}} \put(100,10){\circle{6}}
\put(100,30){\circle{6}} \put(120,10){\circle{6}}
\put(120,30){\circle{6}}

\thinlines \put(3,3){\vector(1,1){14}} \put(4,20){\vector(1,0){12}}
\put(3,37){\vector(1,-1){14}} \put(24,18){\vector(2,-1){12}}
\put(24,22){\vector(2,1){12}}

\put(84,2){\vector(2,1){12}} \put(84,22){\vector(2,1){12}}
\put(84,38){\vector(2,-1){12}}

\put(104,10){\vector(1,0){12}} \put(103,13){\vector(1,1){14}}
\put(104,30){\vector(1,0){12}} \put(103,27){\vector(1,-1){14}}

\put(19,19){\scriptsize {$A$}} \put(98,9){\scriptsize {$A''$}}
\put(98,29){\scriptsize {$A'$}}

\put(8,10){\small {$b$}} \put(8,21){\small {$a$}} \put(10,31){\small
{$a$}} \put(29,16){\small {$d$}} \put(28,26){\small {$c$}}

\put(88,5.5){\small {$b$}} \put(88,26){\small {$a$}}
\put(89,36){\small {$a$}} \put(109,6){\small {$d$}}
\put(114,17.5){\small {$d$}} \put(111.5,24){\small {$c$}}
\put(109,31){\small {$c$}}

\end{picture}\\
Pic. 3. Splitting the states of the automaton $\mathcal{B}.$

\end{figure}

The final version of obtained automaton $\mathcal{B}$ we will call
an \emph{automaton consistent with $\mathcal{A}.$}

Now we are ready to define a no-return random walk on $\mathcal{B}$
as it was claimed above. Namely, let $\mathcal{B}$ be consistent
with $\mathcal{A}$ and let $S$ be a state in $\mathcal{B}$. Denote
by $\nu = \nu(S^{pr})$ the number of edges exiting from the
prototype state $S^{pr}$ in $\mathcal{A}$. The walker moves from $S$
along some outgoing edge with the uniform probability
$\frac{1}{\nu}.$ In this event, the probability
 that the walker hits an element $w \in R$ in $|w|$
steps (when starting at $[\varepsilon]$) is the product of
frequencies of arrows in a direct path from the initial state $I$ to
the accept state $Z$ with the label $w.$
This gives rise to the measure $\lambda_L$ on $R$:
$$\lambda_L(R)  =  \sum_{w \in R}\lambda_L(w) = \sum_{k=0}^{\infty}f^{\prime}_k(R,L),$$
where $$f^{\prime}_k(R,L) = \sum_{w \in R \cap S_k} \lambda_L(w).$$
\par Note that, generally speaking, $f^{\prime}_k(R,L)$ differs from
$f_k(R,L)$ defined in section \ref{subsec:comp-Schreier}. Indeed,
walking in $\mathcal{B}$ we have different number of possibilities
to continue our walk on the next step depending on way we chose. On
the other hand, $f^{\prime}_k(R,F(X)) = f_k(R,F(X)).$
\par Now we can use the tools of random walks to compute $\lambda_L(R)$. Notice,
that $\lambda_L$ is multiplicative, i.e.,
$$\lambda_L(u v) = \lambda_L(u) \lambda_L(v)$$
for any $u,v \in R$ such that $uv = u \circ v$ and $uv \in R$.
  We say that $R$ is {\em $\lambda_L$-measurable}, if $\lambda_L(R)$ is finite.
   A set $R$ is termed {\em exponentially $\lambda_L$-measurable}) if
   $f_k^{\prime}(R,L) \leq q^k$ for all sufficiently large $k$.

The following result is simple, but useful.

Let $w \in F.$ The set  $C_L(w) = L \cap C(w)$ is called an
$L$-cone.  Obviously, $C_L(w)$  is a regular set. We say that
$C_L(w)$ is \emph{$L$-small}, if it is exponentially
$\lambda_L$-measurable.

The following is the main result of this section.

\begin{theorem}\label{Lcones}
Let $R$ be a regular subset of a prefix-closed regular set $L$ in a
free group $F.$ Then either the  prefix closure $\overline{R}$ of
$R$ in $L$ contains a non-small  $L-$cone or $\overline{R}$ is
exponentially $\lambda_L$-measurable.
\end{theorem}

Before proving the theorem we establish a few preliminary facts. We
fix a prefix-closed regular subset $L$ of $F$.

\begin{proposition}
\label{le:basic-Lsparse} Let $R_1$ and $R_2$ be subsets of $F$. Let
also $P$ be one of the properties $\{$ ''to be $L-$measurable'',
''to be exponentially $L-$negligible'', ''to be
$\lambda_L-$measurable'', ''to be exponentially
$\lambda_L-$measurable''$\}.$ Then the following hold:
 \bi
   \item[1)] If $R_1 \subseteq R_2$ and $R_2$ has property $P$ then so is $R_1$.

    \item [2)] If $R_1, R_2$  have property $P$ then so is $R_1 \cup R_2$.

\item [3)]  If $R_1$ and $R_2$ have property $P$ then so is the set
    $$R_1 \circ R_2 = \{r_1r_2 \mid r_i \in R_i, \ c(r_1,r_2) = 0 \}.$$

 \ei

\end{proposition}
  \begin{proof}
The proofs are easy.
\end{proof}

To strengthen the last statement in Proposition \ref{le:basic-Lsparse}  we need the following notation.
For a subset $T \subseteq F$ put $T^\circ_1 = T$ and define recursively  $T^\circ_{k+1} = T^\circ_k \circ T$. Denote
$$T^\circ_{\infty} = \bigcup_{k = 1}^\infty T^\circ_k.$$

\begin{lemma} \label{le:T}
Let $T$ be a regular set and a number $q$,   $0 < q < 1$, such that
$f_k^\prime(T,L) \leq q^k$ for every positive integer $k$. Then the set
$T^\circ_\infty$  is exponentially $\lambda_L-$measurable.
\end{lemma}
\begin{proof}

  Every word $w \in T^\circ_\infty$ of length $k$ comes in the form $w = u_1\circ u_2\circ \ldots \circ u_t$, where $u_i$'s
  are non-trivial elements from $T$ and $k = |u_1| + \ldots +|u_t|$.  On the other hand, if $k = k_1 + \ldots +k_t$ is an arbitrary
  partition of $k$ into a sum of positive integers and $u_1, \ldots, u_t$ are words in $T$ such that $u_i = k_i$,
  then $w = u_1\ldots u_t \in T^\circ_\infty$. Since $\lambda_L$ is multiplicative every partition of $k$ adds to $f^{\prime}_k(T^\circ_\infty,L)$ a
  number $f^{\prime}_{k_1}(T^\circ_\infty,L) \ldots f^{\prime}_{k_t}(T^\circ_\infty,L)$, which is bounded from above by $q^{k_1 + \ldots +k_t} = q^k$.
  If $p(k)$ is the  number of all  partitions of $k$ into a sum of positive integers then $f^{\prime}_k(T^\circ_\infty,L) \leq p(k)q^k$.
   It is known (Hardy and Ramanujan) that
$$p(k) \sim \fracd{e^{\pi\sqrt{\frac{2k}{3}}}}{4k\sqrt{3}}.$$
Hence  $f^{\prime}_k(T^\circ_\infty,L) < q_1^k,$ for some $0 < q <
q_1 < 1$ and all sufficiently large $k$,  so  $T^\circ_\infty$ is
exponentially $\lambda_L-$measurable, as claimed.
\end{proof}

\medskip \noindent {\it Proof  of Theorem \ref{Lcones}}.
In the most part we follow the proof of Theorem
\ref{th:regular-negl-thick} from \cite{multiplicative}. Suppose that
all $L$-cones in $\overline{R}$ are non-small.  Since $R \subseteq
\overline{R}$ by Proposition \ref{le:basic-Lsparse} we can assume
that $R$ itself is prefix-closed in $L.$ We have to prove that $R$
is exponentially $\lambda_L$-measurable. Let $R = L(\mathcal{B})$
and $\mathcal{B}$ consistent to $\mathcal{A}$ (where $\mathcal{A}$
recognize $L$).

 It is convenient to further split $\mathcal{B}$ into two parts. Denote by  $\mathcal{B}_1$ the automaton obtained
from $\mathcal{B}$ by removing all arrows exiting from $Z$.

\begin{center}
 \includegraphics[width=5cm]{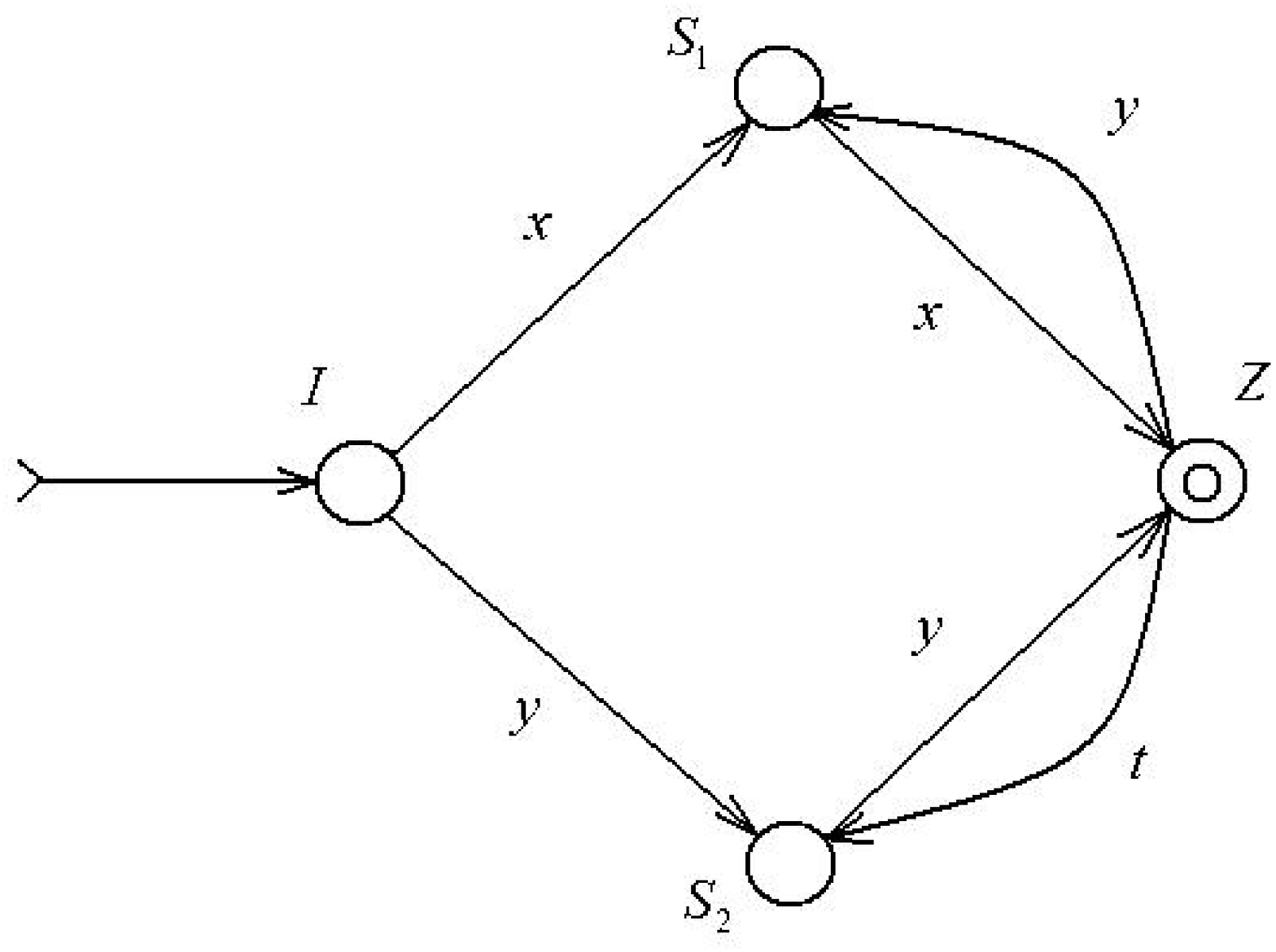}\\
Pic. 4. An automaton $\mathcal{B}.$
\end{center}

\begin{center}
 \includegraphics[width=5cm]{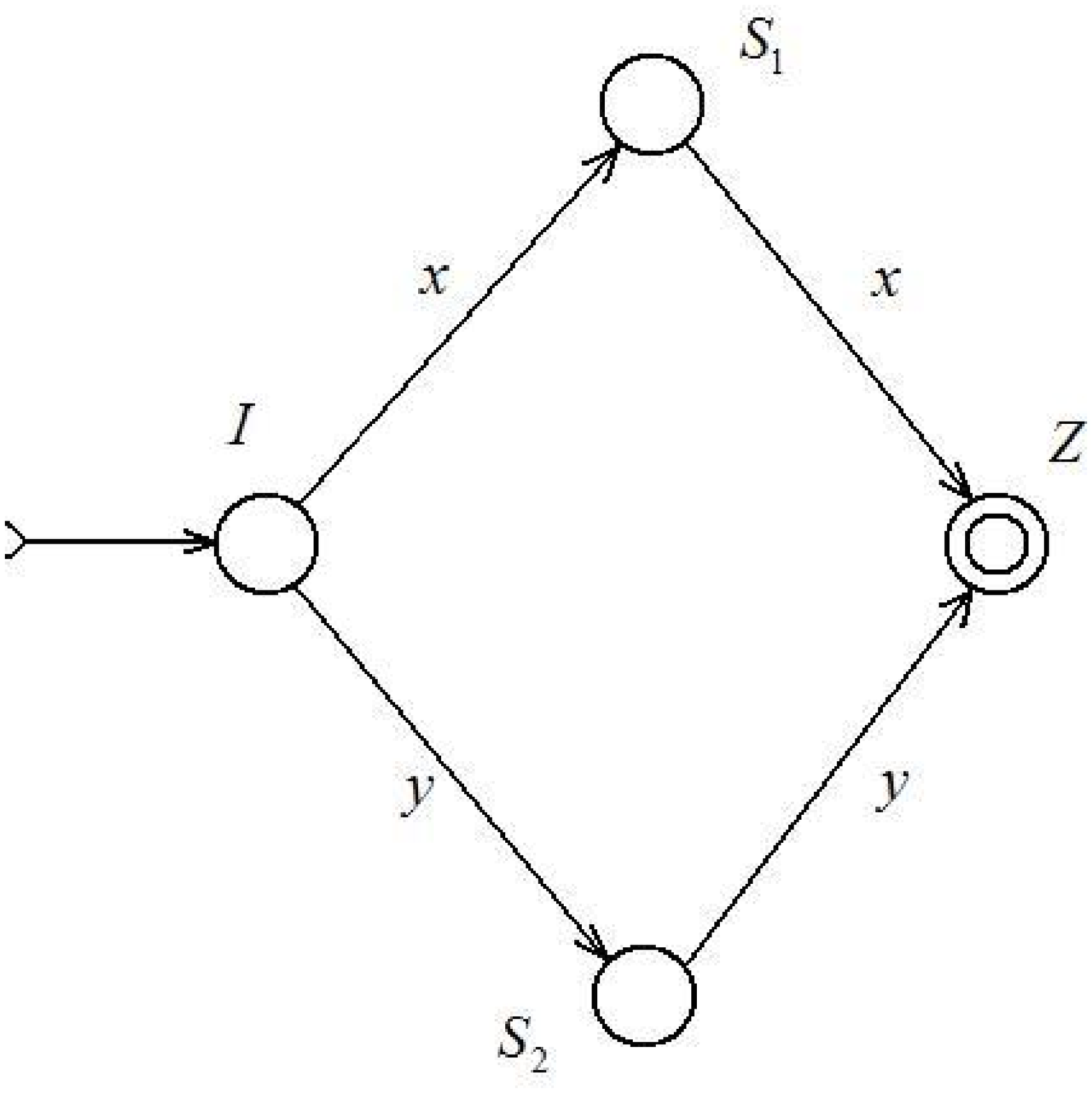}\\
Pic. 5. An automaton $\mathcal{B}_1.$
\end{center}

Let  $\mathcal{B}_2$ be  the automaton formed by all states in
$\mathcal{B}$ that are accessible from the state $Z,$ with the same
arrows between them as in $\mathcal{B};$ $Z$ is the  only initial
and  accepting state of  $\mathcal{B}_2$.

\begin{center}
 \includegraphics[width=5cm]{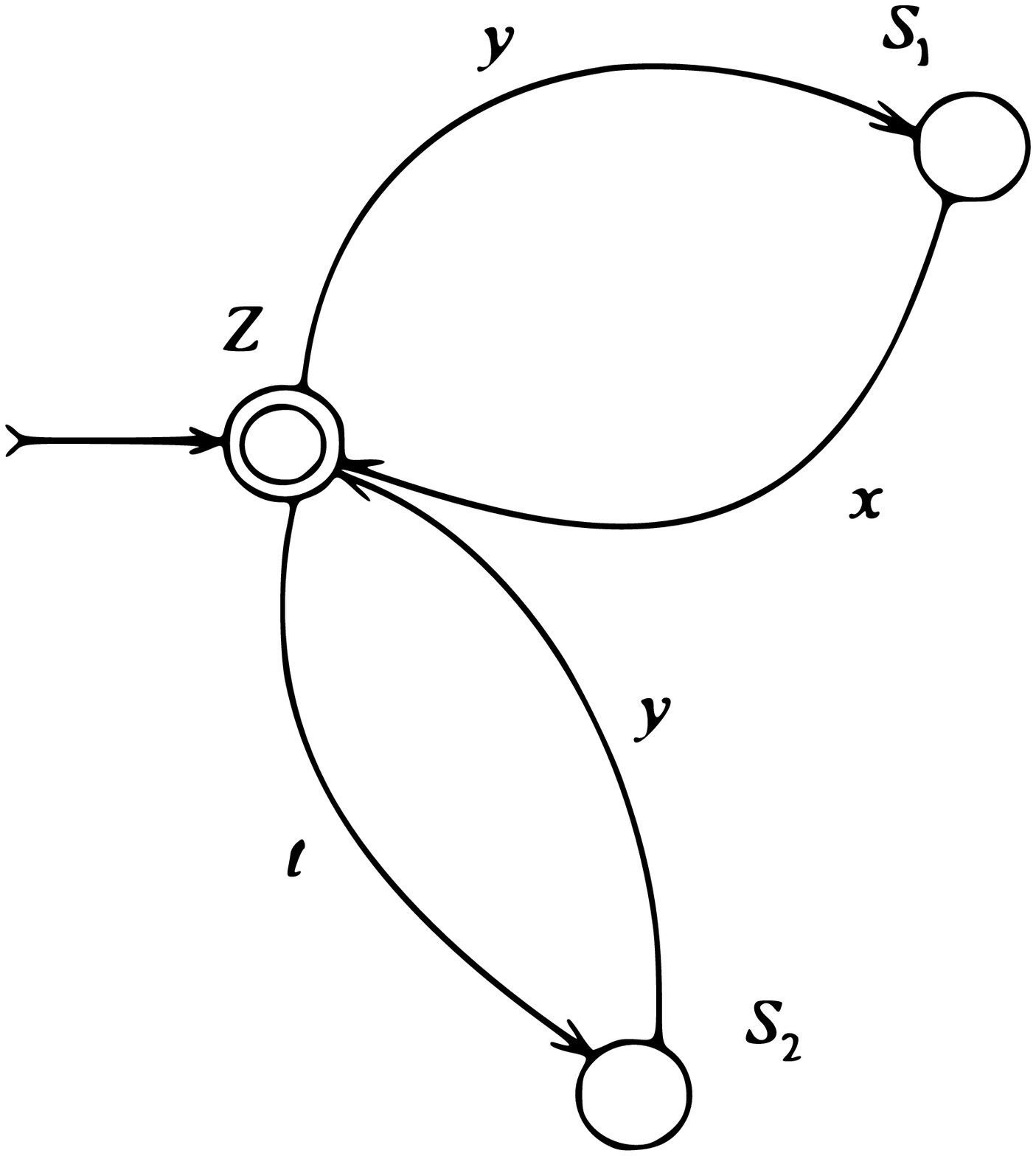}\\
Pic. 6. Automaton $\mathcal{B}_2.$
\end{center}

We assign to arrows in $\mathcal{B}_1$ and $\mathcal{B}_2$ the same
frequencies as to the corresponding arrows in $\mathcal{B}.$ If
$R_1$ and $R_2$ are the languages accepted by $\mathcal{B}_1$ and
$\mathcal{B}_2$ then, obviously, $R = R_1 \circ R_2.$  By
Proposition \ref{le:basic-Lsparse} to prove the theorem it suffices
to show that $R_1$ and $R_2$ are exponentially
$\lambda_L$-measurable.

\medskip \noindent
{\it Claim}. {\it The set $R_2$ is exponentially
$\lambda_L$-measurable.}

\medskip \noindent
{\it Proof of the claim}. Notice, that for every $w \in R_1$ $w
\circ R_2 \subseteq L(\mathcal{A}) = R$  and $w \circ R_2$ is an
$L-$cone. It is easy to see, that $R_2$ is exponentially
$\lambda_L$-measurable if and only if so $w \circ R_2$ is.

Let $R_3 \subseteq R_2$ be the  subset consisting  of  those non-trivial
words $w \in R_2,$ whose paths $p_w$  visit the state $Z$ of
$\mathcal{B}_2$ only once. The set $R_3$ is regular - it is accepted by an automaton
$\mathcal{B}_3$, which is obtained from
$\mathcal{B}_2$ by splitting the state $Z$ into two separate states: the initial
state $Z_1$ and an accepting state $Z_2$, in such a way that  the edges exiting from $Z$ in $\mathcal{B}_2$ are
now exiting from $Z_1$ and there no ingoing edges at $Z_1$, while there are no edges exiting from $Z_2$ and all
those arrows incoming in $Z$ in $\mathcal{B}_2$ are
now incoming into  $Z_2$.
\begin{center}
 \includegraphics[width=6cm]{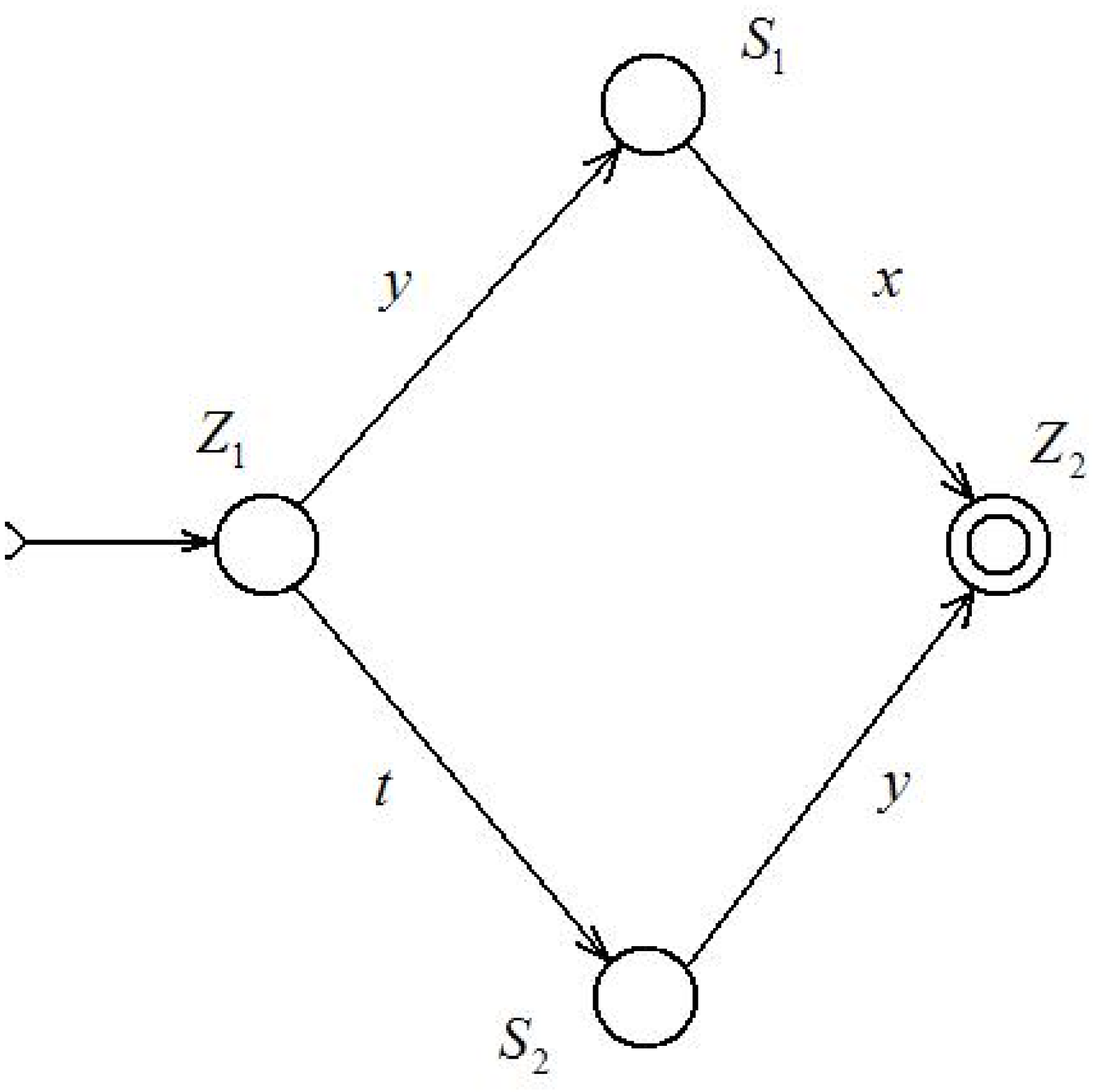}\\
Pic. 7. An automaton $\mathcal{B}_3.$
\end{center}

 It follows
immediately from the  construction, that
$$R_2  = \{\varepsilon\} \bigcup R_3 \bigcup (R_3 \circ R_3) \bigcup (R_3 \circ R_3 \circ R_3) \bigcup  \ldots  = (R_3)^\circ_\infty$$
so
\begin{equation}\label{r2tor3}
\lambda_L(R_2) = \lambda_L(R_3) + (\lambda_L(R_3))^2 +
(\lambda_L(R_3))^3 + \ldots .
\end{equation}

By Proposition \ref{le:T} it suffices to show that there is a number
$q, 0 < q < 1$,  such that $f^{\prime}_k(R_3,L) \leq q^k$ for every
$k$ (not only for sufficiently large $k$). It is not hard to see
that this condition holds if $R_3$ is exponentially
$\lambda_L$-measurable and $\lambda_L(R_3) < 1$, so it suffices to
prove the latter two statements.

By our assumption all $L$-cones in $R = \overline{R}$ are $L-$small.
If for every state $[w]=S$ in $\mathcal{B}_2$ and every $x \in X
\cup X^{-1}$ there is an outgoing edge labeled by $x$ at $[w]$ if
and only if the same holds for the state $S^{pr}$ in $\mathcal{A}$
(i.e.,
 $\mathcal{B}_2$ is $X-$complete relative to $\mathcal{A}$) then for every given $w \in R_1$
 one has $C(w) \cap  \overline{R} = w \circ \overline{R_2} = C(w) \cap L$, so $w \circ \overline{R_2}$ is an $L$-cone.
 Hence it is $L-$small, i.e., exponentially $\lambda_L$-measurable, but then the set $\overline{R_2}$, hence $R_2$, is exponentially
 $\lambda_L$-measurable, as claimed.

This implies that for some state
$S$ there are less then $\nu = \nu(S^{pr})$ arrows exiting from $S$.
 Consider a finite Markov chain
$\mathcal{M}$ with the same states as in $\mathcal{B}_3$ together
with an additional dead state $D$. We set transition probabilities
from $Z_2$ to $Z_2$ and from $D$ to $D$ being equal $1.$ Every arrow
from a state $S$ in $\mathcal{B}_3$ gives the corresponding
transition from the state $S$ in $\mathcal{M}$ which we assign  the
transition probability $\fracd{1}{\nu }.$ If at some state $S$ of
type $x$ in $\mathcal{B}_3$ there is no exiting arrow labeled $y \in
(X \cup X^{-1}) \smallsetminus \{x^{-1}\},$ in $\mathcal{M}$ we make
a transition from $S$ to $D$ with the transition probability
$\fracd{1}{\nu }.$ This describes $\mathcal{M}$.

\begin{center}
 \includegraphics[width=7cm]{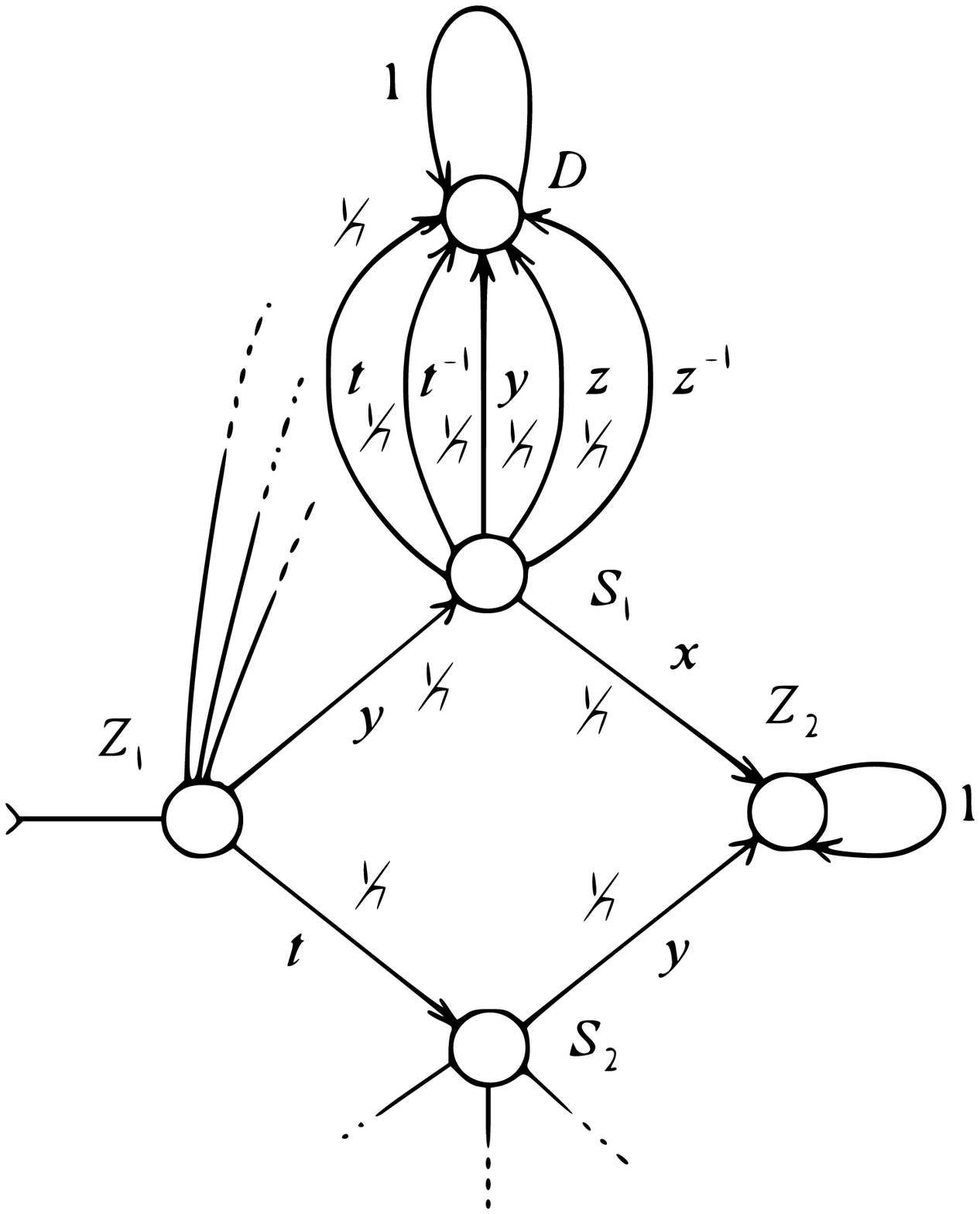}\\
Pic. 8. An automaton $\mathcal{M}.$
\end{center}

The states $Z_2$ and $D$ of Markov chain $\mathcal{M}$ are
absorbing, and all other states are transient. The probability
distribution on $\cal M$ concentrated at the initial state $Z_1$,
converges to the steady state $P$ which is zero everywhere with the
exception of the two absorbing states $Z_2$ and $D$. Obviously,
$P(Z_2) = \lambda_L(R_3)$.  Since $P(D) \ne
0$ we have $\lambda_L(R_3) = P(Z_2) < 1$, so one of the required
conditions on $R_3$ holds (for more details on this proof we refer
to \cite{multiplicative,KS}). The other one follows directly from
Corollary 3.1.2 in \cite{KS}, which claims that in this case $R_3$
is exponentially $\lambda_L$-measurable. This proves the claim.

A similar argument shows that $R_1$ is exponentially
$\lambda_L$-measurable. This proves the theorem.

{\bf Acknowledgements. } The authors thank Alexandre Borovik for
very fruitful discussions.


\end{document}